\documentclass[12pt]{amsart}
\usepackage{amssymb,amsmath,geometry,latexsym,
graphics,tabularx,shapepar,enumerate,mathtools}

\newcommand{\undt}[1]{\underline{t}_{#1}}

\newcommand{\ZZ}{\mathbb{Z}}
\newcommand{\FF}{\mathbb{F}}
\newcommand{\CC}{\mathbb{C}}
\newcommand{\QQ}{\mathbb{Q}}
\newcommand{\EE}{\mathbb{E}}

\newcommand{\NN}{\mathbb{N}}

\newcommand{\TT}{\mathbb{T}}

\usepackage[OT2,T1]{fontenc}
\DeclareSymbolFont{cyrletters}{OT2}{wncyr}{m}{n}
\DeclareMathSymbol{\Sha}{\mathalpha}{cyrletters}{"58}

\newtheorem{Theorem}{Theorem}
\newtheorem{Lemma}[Theorem]{Lemma}
\newtheorem{Proposition}[Theorem]{Proposition}

\newtheorem{Definition}[Theorem]{Definition}
\newtheorem{Remark}[Theorem]{Remark}

\date{March, 2017}

\title[Sum-shuffle formula]{A sum-shuffle formula for zeta values in Tate algebras}

\author{F. Pellarin}
\address{Institut Camille Jordan, UMR 5208 Site de Saint-Etienne, 23 rue du Dr. P. Michelon, 42023 Saint-Etienne,
France}
\email{federico.pellarin@univ-st-etienne.fr}
\keywords{Multiple zeta values, Function field arithmetic, Carlitz zeta values}
\thanks{This project has received funding from the European Research Council (ERC)
under the European Union's Horizon 2020 research and innovation programme under
the Grant Agreement No 648132.}

\begin{document}

\maketitle

\renewcommand{\abstractname}{Résumé}

\begin{abstract}

Nous démontrons une formule de mélange pour des valeurs zêta multiples dans des algèbres de Tate (en caractéristique non nulle) introduites dans \cite{PEL3}. Ce résultat se déduit d'un résultat analogue pour les sommes de puissances tordues et implique que le $\FF_p$-espace vectoriel des valeurs zêta multiples dans les algèbres de Tate est une $\FF_p$-algèbre.
\end{abstract}

\renewcommand{\abstractname}{Abstract}

\begin{abstract}
We prove a sum-shuffle formula 
for multiple zeta values in Tate algebras (in positive characteristic), introduced in \cite{PEL3}.
This follows from an analog result for double twisted power sums, implying 
that an $\FF_p$-vector space generated by multiple zeta values in Tate algebras is an $\FF_p$-algebra.
\end{abstract}

\maketitle

\section{Introduction}

Let $A=\FF_q[\theta]$ be the ring of polynomials in an indeterminate $\theta$  with coefficients in $\FF_q$ the finite field with 
$q$ elements and characteristic $p$, and let $K$ be the fraction field of $A$. We consider variables $t_i$
independent over $K$, for all $i \in\NN^*:=\{1,\ldots\}$ the set of positive natural numbers. For 
$\Sigma\subset\NN^*$ a finite subset, we denote by $\undt{\Sigma}$ the collection of variables
$(t_i)_{i\in\Sigma}$, so that $\FF_q(\undt{\Sigma})$, $K(\undt{\Sigma})$ denote the 
fields  $\FF_q(t_i:i\in\Sigma)$ etc.

We denote by $A^+$ the multiplicative monoid of monic polynomials of $A$ (in $\theta$) and,
for $d\geq 0$ an integer, we denote by $A^+(d)$ the subset of monic polynomials of $A$
of degree $d$. With $\Sigma\subset\NN^*$ a finite subset and for all $i\in\Sigma$, we denote by $\chi_{t_i}:A\rightarrow\FF_q[\underline{t}_\Sigma]$
the unique $\FF_q$-linear map which sends $\theta$ to $t_i$ (the notation does not reflect dependence on $\Sigma$ to avoid unnecessary complication). More generally,
we denote by $\sigma_\Sigma$ the 
semi-character $$\sigma_\Sigma:A^+\rightarrow\FF_q[\underline{t}_\Sigma]$$
defined by $\sigma_\Sigma(a)=\prod_{i\in\Sigma}\chi_{t_i}(a)$.
The associated {\em twisted power sum} of order $k$ and degree $d$ is the sum:
$$S_d(k;\sigma_\Sigma)=\sum_{a\in A^+(d)}a^{-k}\sigma_\Sigma(a)\in K[\underline{t}_\Sigma].$$
If $\Sigma=\emptyset$ we recover the power sums already studied by several authors; see Thakur's \cite{THA}
and the references therein. For general $\Sigma$ these sums have been the object of 
study, for example, in the papers \cite{ANG&PEL,DEM}.
These twisted power sums are also the basic tools to construct certain zeta values in Tate algebras,
see for example \cite{ANG&PEL,ANG&PEL2,APTR,PEL2}. We set:
$$\zeta_A(n;\sigma_\Sigma):=\sum_{d\geq 0}S_d(k;\sigma_\Sigma)\in \TT_\Sigma(K_\infty),$$
where, for $L$ a complete valued field, $\TT_\Sigma(L)$ denotes the completion of the polynomial
ring $L[\underline{t}_\Sigma]$ for the Gauss valuation.
In \cite{ANG&PEL2,APTR}, the following result is proved, where $|\Sigma|$ denotes the 
cardinality of $\Sigma$:
\begin{Theorem}\label{anglespellarin}
If $|\Sigma|\equiv1\pmod{q-1}$ and $s=|\Sigma|>1$, there exists a polynomial $\lambda_{1,\Sigma}\in A[\undt{\Sigma}]$, monic of degree $r:=\frac{s-q}{q-1}$ in $\theta$, such that
$$\zeta_A(1;\sigma_\Sigma)=(-1)^{\frac{s-1}{q-1}}\frac{\widetilde{\pi}\lambda_{1,\Sigma}}{\prod_{i\in\Sigma}\omega(t_i)},$$
where $\widetilde{\pi}$ is a fundamental period of Carlitz module, and $\omega$ denotes the 
Anderson-Thakur function.
\end{Theorem}
We also recall from \cite{PEL2} the formula
\begin{equation}\label{11}
\zeta_A(1;\chi_t)=\frac{\widetilde{\pi}}{(\theta-t)\omega(t)}
\end{equation} which complements Theorem 
\ref{anglespellarin} in the case $\Sigma=\{1\}$ (and $t=t_1$).
One captivating peculiarity of the above formulas is that they constitute a bridge to a class of quite simple, although apparently different objects. For example, 
the formula (\ref{11}) can be rewritten (see \cite{PEL2}) as
$$\zeta_A(1,\chi_t)=\prod_{i>0}\frac{1-\frac{t}{\theta^{q^i}}}{1-\frac{\theta}{\theta^{q^i}}},$$
and shows that $\zeta_A(1,\chi_t)$ is the reduction modulo $p$ of a simple 
formal series in $1+\frac{1}{\theta}\ZZ[t][[\frac{1}{\theta}]]$ (a "Mahler's series").
The dependency of $q$, although unavoidable, is very transparent, as it is only involved in the "raising to the power $q$"
process, while this is not necessarily visible at the first sight in the initial definition of $\zeta_A(1,\sigma_\Sigma)$, and neither is in its retranscription as an Eulerian product
$$\zeta_A(1,\sigma_\Sigma)=\prod_P\left(1-\frac{\chi_t(P)}{P}\right)^{-1},$$ running over the irreducible polynomials of $A^+$.

A similar phenomenon holds in the case $s=|\Sigma|>1$ of Theorem \ref{anglespellarin} for $q$ big enough.
The polynomial $\lambda_{1,\Sigma}$ itself, for example, is a polynomial monic in $\theta$ of degree $r$ dependent of $q$ in a very simple way (\footnote{The arithmetic properties of $\lambda_{1,\Sigma}$ still remain deep and mysterious, as pointed out for example in the papers \cite{APTR,ANDTR}, and the dependence becomes more unpredictable if $q$ is small.}),
as we have already pointed out, and 
it is easy to deduce, from the elegant arguments presented in \cite{ANDTR}
(see also the table therein), that 
the coefficient of $\theta^{r-1}$ in $\lambda_{1,\Sigma}$ is equal, for $q$ big enough, to the reduction modulo $p$ of
$$
-e_{s-q+1}(\undt{\Sigma})-e_{s-2(q-1)}(\undt{\Sigma})-\cdots-e_{q}(\undt{\Sigma})-\sum_{i\in\Sigma}t_i\in\ZZ[\undt{\Sigma}],$$
where $e_n(\undt{\Sigma})$ is the elementary symmetric polynomial of degree $n$ in the variables
$\undt{\Sigma}$ (this for $s\geq 2q-1$ and $s\equiv1\pmod{q-1}$, while in the case $s=q$, we have $\lambda_{1,\Sigma}=1$). These phenomena ultimately
arise because certain universality features of sequences of power sums hold, and more generally, the same principles govern the behavior of twisted power sums (see \cite{THA} and \cite{DEM}). 

In this paper, we analyze another aspect of the above principles.
We shall show (Theorem \ref{corsumshuffle}) that any product $\zeta_A(1,\sigma_U)\zeta_A(1,\sigma_V)$ ($U,V\subset\NN^*$, $U\cap V=\emptyset$) of such zeta values satisfies a {\em sum-shuffle product formula,} and this will be again deduced from 
properties of twisted power sums (Theorem \ref{simplesumshuffle}).
We deduce that a certain $\FF_p$-vector space of multiple zeta values
in Tate algebras also has a structure of $\FF_p$-algebra (Theorem \ref{conjecture2}).
The formula of Theorem \ref{corsumshuffle} is submitted to universal rules very similar to those of the above remarks: we can say, loosely, that they "almost lift to characteristic zero". 
\section{The result}

Before presenting the results, we have to now introduce multiple twisted power sums and multiple zeta values in our context.

\begin{Definition}\label{defpowersums}{\em 

Let $\Sigma\subset\NN^*$
be a finite subset. If $U,V$ are two subsets of $\Sigma$ such that 
$U\cap V=\emptyset$, we denote by $U\sqcup V$ their union.
Now, suppose that for an integer
$r>0$, we have subsets $U_i$ ($i=1,\ldots,r$) such that
$\Sigma=U_1\sqcup\cdots\sqcup U_r$. Further, let $d$ be a non-negative integer. We have the {\em multiple twisted power sum} of degree $d$ associated to this data:
$$S_d\left(\begin{matrix}\sigma_{U_1} & \sigma_{U_2} & \cdots & \sigma_{U_r}\\
n_1 & n_2 & \cdots & n_r\end{matrix}\right)=
S_{d}(n_1;\sigma_{U_1})\sum_{d> i_2> \cdots> i_r\geq 0}
S_{i_2}(n_2;\sigma_{U_2})\cdots S_{i_r}(n_r;\sigma_{U_r})\in K[\undt{\Sigma}].$$
The integer $\sum_in_i$ is called the 
{\em weight} and the integer $r$ is called its {\em depth}.}\end{Definition}
We can write in both ways $S_d(n;\sigma_\Sigma)=S_d\binom{\sigma_\Sigma}{n}$. Observe also that, if $\Sigma=\emptyset$, then $\sigma_\Sigma=\boldsymbol{1}$ the trivial semi-character.
$$S_d\left(\begin{matrix}\boldsymbol{1} & \boldsymbol{1} & \cdots & \boldsymbol{1}\\
n_1 & n_2 & \cdots & n_r\end{matrix}\right)=S_d(n_1,n_2,\ldots,n_r)\in K,$$
in the notations of Thakur. 

With $n_1,\ldots,n_r\geq 1$ and the semi-characters $\sigma_{U_1},\ldots,\sigma_{U_r}$
as above, we introduce the associated {\em multiple zeta value}
$$\zeta_A\left(\begin{matrix}\sigma_{U_1} & \sigma_{U_2} & \cdots & \sigma_{U_r}\\
n_1 & n_2 & \cdots & n_r\end{matrix}\right):=\sum_{d\geq 0}S_d\left(\begin{matrix}\sigma_{U_1} & \sigma_{U_2} & \cdots & \sigma_{U_r}\\
n_1 & n_2 & \cdots & n_r\end{matrix}\right)\in \TT_\Sigma(K_\infty).$$ The sum thus converges in the Tate algebra $\TT_\Sigma(K_\infty)$
where $K_\infty$ is the completion $\FF_q((1/\theta))$ of $K$ at the infinity place.
Explicitly, we have:
\begin{equation}\label{definitionmzv}
\zeta_A\left(\begin{matrix}\sigma_{U_1} & \sigma_{U_2} & \cdots & \sigma_{U_r}\\
n_1 & n_2 & \cdots & n_r\end{matrix}\right)=\sum_{d\geq 0}\sum_{a_1,\ldots,a_r\in A^+
\atop d=\deg_\theta(a_1)>\cdots>\deg_\theta(a_r)\geq0}\frac{\sigma_{U_1}(a_1)\cdots \sigma_{U_r}(a_r)}{a_1^{n_1}\cdots a_r^{n_r}}.\end{equation}
Again, the integer $\sum_in_i$ is called the 
{\em weight} of the above multiple zeta value and the integer $r$ is called its {\em depth}. These elements of the above considered Tate algebras
have been introduced  and first discussed in \cite{PEL3}.

\subsection{Non-vanishing of our multiple zeta values} 

To ensure that our multiple zeta values generate a non-trivial theory, we must now prove that 
they are not identically zero; this is the purpose of Proposition \ref{twists} below. The tools we use will be also crucial in other parts of the paper.

In \cite[Proposition 4]{PEL3} it was proved that the multiple zeta values (\ref{definitionmzv}) in $\TT_\Sigma(K_\infty)$, seen as functions of the variables $t_i\in\CC_\infty$ (for $i\in\Sigma$), where 
$\CC_\infty$ denotes the completion of an algebraic closure of $K_\infty$, extend to entire functions 
$\CC_\infty^{|\Sigma|}\rightarrow\CC_\infty$. We denote by $\EE_{\Sigma}(K_\infty)$ the sub-$K_\infty$-algebra of 
$\TT_\Sigma(K_\infty)$ whose elements extend to entire functions as above, so that all the multiple zeta values as in (\ref{definitionmzv}) belong to this sub-algebra. We also denote by 
$\tau:\TT_\Sigma(K_\infty)\rightarrow\TT_\Sigma(K_\infty)$ the unique continuous, open $\FF_q[\underline{t}_\Sigma]$-linear endomorphism which reduces to the map $c\mapsto c^q$ when restricted over $K_\infty$. 
Then, $\tau$ induces an $\FF_q[\underline{t}_\Sigma]$-linear endomorphism of $\EE_{\Sigma}(K_\infty)$.

We further have:
\begin{Proposition}\label{twists} Let us consider $\Sigma\subset\NN^*$ a finite subset as above, and subsets
$U_1,\ldots,U_r$ such that $\Sigma=U_1\sqcup\cdots\sqcup U_r$. If $j\in\Sigma$, we write $i_j$
for the unique integer $i\in\{1,\ldots,r\}$ such that $j\in U_i$.
Let us consider $n_1,\ldots,n_r$ positive integers, and let us denote by $f$ the multiple zeta value in  (\ref{definitionmzv}). Let $N$ be a non-negative integer. Let us also consider, for all $i=1,\ldots,r$ and $j\in U_i$, non-negative integers $k_{i,j}$ (hence, $i=i_j$).
We suppose that for all $i=1,\ldots,r$, 
\begin{equation}\label{conditionNk}
q^Nn_i>\sum_{j\in U_i}q^{k_{i,j}}.\end{equation} 
Then, the evaluation 
$$\tau^N(f)_{\begin{smallmatrix}t_j=\theta^{q^{k_{i_j,j}}}\\ j\in\Sigma\end{smallmatrix}}\in K_\infty,$$
well defined, is equal to the multiple zeta value of Thakur
$$\zeta_A\left(q^Nn_1-\sum_{j\in U_1}q^{k_{1,j}},\ldots,q^Nn_r-\sum_{j\in U_r}q^{k_{r,j}}\right).$$
In particular, the multiple zeta values as in (\ref{definitionmzv}) are all non-zero.
\end{Proposition}
\begin{proof} 
Here and in the following, the evaluation is operated after the application of the operator $\tau^N$ (note that these operations do not commute).
Since $f$ is entire by the remarks preceding the proposition, the evaluation is well defined in $K_\infty$ independently of the hypothesis on $N$ and the $k_{i,j}$'s. That the evaluation is a multiple zeta value of Thakur follows from the mentioned conditions, observing that if $a\in A=\FF_q[\theta]$, then 
$a(\theta^{q^k})=a(\theta)^{q^k}$ for all $k\in \ZZ$. 

Thakur's multiple zeta values are known to be non-zero (see \cite[Theorem 4]{THA}).
In particular, setting $k_{i,j}=0$ for all $i,j$, it is always possible to find $N$ such that 
(\ref{conditionNk}) holds. This implies that $f$, and the multiple zeta values as in (\ref{definitionmzv}) are all non-zero.
\end{proof}
\subsubsection{Example}\label{nonvanishingexample} In the case of $r=1$, we have that, for all $N\geq 0$ and $k_i\geq 0$ for $i\in\Sigma$ such that $q^N-\sum_{i\in\Sigma}q^{k_i}>0$,
\begin{equation}\label{evalrone}
\tau^N(\zeta_A(1,\sigma_\Sigma))_{\begin{smallmatrix}t_i=\theta^{q^{k_i}}\\
i\in\Sigma\end{smallmatrix}}=\zeta_A\left(q^N-\sum_{i\in\Sigma}q^{k_i}\right)\in K_\infty.\end{equation}
Since for each $m\in\NN^*$ there exists $N,k_i\geq 0$ (for $i\in\Sigma$) with $m=q^N-\sum_{i\in\Sigma}q^{k_i}$, there also exists, for $m$ given, $\Sigma\subset\NN^*$ a finite subset (actually, infinitely many finite such subsets) such that the Carlitz zeta value $\zeta_A(m)\in K_\infty$ comes from an evaluation 
of $\zeta_A(1,\sigma_\Sigma)$ of the same type as in (\ref{evalrone}). More generally, a similar property holds for the multiple zeta values of Thakur.

\subsection{The result}

We prove, in this paper, a sum shuffle formula for products $$\zeta_A(1,\sigma_U)\zeta_A(1,\sigma_V),$$ with $U\sqcup V=\Sigma\subset\NN^*$:
\begin{Theorem}\label{corsumshuffle}
The following formula holds, for all $\Sigma\subset\NN^*$ and $U\sqcup V=\Sigma$:
\begin{multline*}
\zeta_A\left(\begin{matrix} \sigma_U \\
1\end{matrix}\right)\zeta_A\left(\begin{matrix} \sigma_V \\
1\end{matrix}\right)-\zeta_A\left(\begin{matrix}\sigma_\Sigma \\
2\end{matrix}\right)= \\ \zeta_A\left(\begin{matrix}\sigma_U & \sigma_V \\
1 & 1\end{matrix}\right)+\zeta_A\left(\begin{matrix}\sigma_V & \sigma_U \\
1 & 1\end{matrix}\right)-\sum_{\begin{smallmatrix}I\sqcup J= \Sigma \\
|J|\equiv1\pmod{q-1}\\
J\subset U\text{ or }J\subset V\end{smallmatrix}}\zeta_A\left(\begin{matrix}\sigma_{I} & \sigma_{J} \\
1 & 1\end{matrix}\right).\end{multline*}

\end{Theorem}
The reader will notice the universality phenomenon mentioned above:
the coefficients of the right-hand side of the above formula are $0,1,-1$
and they are determined upon a simple divisibility by $q-1$ condition, and 
the position $I,J$ of the subsets of $\Sigma$ relative to $U,V$. 
We will use techniques of Thakur in \cite{THA2} and some 
linear algebra over $\FF_p$ to show the existence of
the coefficients. To compute them, as it is easily verified that they are uniquely determined, we will use the sum-shuffle formula
of Chen \cite{CHE}, which also arises in several ways from our formula by specialization. We recall this result.
\begin{Theorem}[Chen]
For all $n,m>0$,
$$\zeta_A(m)\zeta_A(n)-\zeta_A(m+n)=\sum_{0<j<m+n\atop q-1|j}f_j\zeta_A(m+n-j,j),$$
where $$f_j=(-1)^{m-1}\binom{j-1}{m-1}+(-1)^{n-1}\binom{j-1}{n-1}.$$
\end{Theorem}
It is further possible to deduce, by quite standard methods (applying 
a twisted Frobenius endomorphism a certain amount of times and specializing some variables), the following:
 \begin{Theorem}\label{conjecture2}
For all $\Sigma\subset\NN^*$ a finite subset, the $\FF_p$-subvector space of $\TT_\Sigma(K_\infty)$ generated by the multiple zeta values
(\ref{definitionmzv}) is an $\FF_p$-algebra.\end{Theorem}

\section{sum-shuffle relations}

We follow the main idea of Thakur in \cite{THA2}, where he proves a 
sum-shuffle formula for the product of two power sums and he deduces
from this result that the $\FF_p$-sub-vector space 
of $K_\infty$ generated by his multiple zeta values $\zeta_A(n_1,\ldots,n_k)$ is an $\FF_p$-algebra
(this is the case $\Sigma=\emptyset$). Thakur's result thus relies in 
universal families of sum-shuffle relations for power sums products that he proves 
by reducing to the case of one-degree power sums (the case of $d=1$), a technique which is also naturally suggested
by the philosophy of ``solitons". We follow the principles of this proof. 
The main difference between this part of our proof and Thakur's 
is situated in the case of $d=1$, which presents 
new structures.

\subsection{Sum-shuffle formulas for power sums}

We shall prove:
\begin{Theorem}\label{simplesumshuffle}
Let $U,V$ be subsets of $\Sigma$ such that $U\sqcup V=\Sigma$. Then, the following formula holds
$$S_d\left(\begin{matrix} \sigma_U \\ 1 \end{matrix}\right)
S_d\left(\begin{matrix} \sigma_V \\ 1 \end{matrix}\right)-
S_d\left(\begin{matrix} \sigma_\Sigma \\ 2 \end{matrix}\right)=-\sum_{\begin{smallmatrix}I\sqcup J=\Sigma
\\
|J|\equiv1\pmod{q-1}\\
J\subset U\text{ or }J\subset V\end{smallmatrix}}S_d\left(\begin{matrix} \sigma_{I} & 
\sigma_J\\ 1 & 1 \end{matrix}\right),\quad d\geq 0.$$
\end{Theorem}
Theorem \ref{corsumshuffle} easily follows by taking the sum for $d\geq 0$.
The identity is clearly satisfied if $d=0$. We first develop some tools 
involving certain vector spaces generated by multiples of twisted power sums, then
we
show the identity for $d=1$.

\subsubsection{$\FF_p$-subvector spaces of twisted power sums}\label{thespacesvv}
We set $[1]=\theta^q-\theta$.
We note that 
$$\frac{[1]}{\theta-\lambda}=\frac{\theta^q-\theta}{\theta-\lambda}=\frac{\theta^q-\lambda^q+\lambda-\theta}{\theta-\lambda}=(\theta-\lambda)^{q-1}-1,\quad \forall \lambda\in\FF_q.$$
Hence, for all $U\subset\Sigma$, using that $[1]=\prod_{\lambda\in\FF_q}(\theta-\lambda)$,
$$P_U:=[1]S_1(1,\sigma_U)=[1]\sum_{\lambda\in\FF_q}\frac{\prod_{i\in U}(t_i-\lambda)}{\theta-\lambda}=
\sum_{\lambda\in\FF_q}((\theta-\lambda)^{q-1}-1)\prod_{i\in U}(t_i-\lambda)\in A[\underline{t}_U].$$
In fact it is, more precisely, a polynomial of $\FF_p[\theta][\underline{t}_U]$ of degree $\leq q-1$ in $\theta$ (\footnote{In fact, it can be proved that the degree in $\theta$ of $P_{\Sigma}$ is exactly $q-1$ if $|\Sigma|\geq q$. For this, one can apply the formula (\ref{S0}) and the arguments following it. Since this will not be used in this paper, we will not give full details about this.}). The claim on the degree in $\theta$ being clear, we indeed observe that $P_V$ is invariant, by construction,
under the action of $\operatorname{Gal}(\FF_p^{ac}/\FF_p)$.
Let $\mathcal{V}_\Sigma$ be the $\FF_p$-subvector space of $\FF_p[\theta][\underline{t}_\Sigma](<q)$ (a shortcut for polynomials 
of degree $<q$ in $\theta$)
generated by the polynomials $P_U$ with $U\subset\Sigma$. 
We have that $\mathcal{V}_{\Sigma'}\subset \mathcal{V}_\Sigma$ if $\Sigma'\subset\Sigma$.
In particular, $\mathcal{V}_\emptyset=1\cdot\FF_p$. In Lemma \ref{basis}
we will show that the $\FF_p$-vector space $\mathcal{V}_\Sigma$ has dimension $2^{|\Sigma|}$
but we do not need this information right now.

Let $(\mathcal{D}_n)_{n\geq 0}$ be the system of
higher derivatives of $\FF_p[\theta][\underline{t}_\Sigma]$ in $\theta$ which is $\FF_p[\underline{t}_\Sigma]$-linear and such that $\mathcal{D}_n(\theta^m)=\binom{m}{n}\theta^{m-n}$. 
The main result of this subsection is the following:
\begin{Proposition}\label{dstability}
For all $\Sigma$, the space $\mathcal{V}_\Sigma$ is stable under the higher derivations $(\mathcal{D}_n)_{n\geq 0}$.
\end{Proposition}
The proof of this proposition occupies the rest of this subsection. We note, since 
$\mathcal{V}_\Sigma\subset \FF_p[\theta][\underline{t}_\Sigma](<q)$, that it suffices to show that
$$\mathcal{D}_n(\mathcal{V}_\Sigma)\subset \mathcal{V}_\Sigma,\quad n=1,\ldots,q-1.$$ 
We define, for all $n\geq 0$:
$$\mathcal{V}_\Sigma^{(n)}=\operatorname{Vect}_{\FF_p}(\mathcal{V}_U:|\Sigma\setminus U|\geq n).$$
Hence, we have $$\mathcal{V}_\Sigma=\mathcal{V}_\Sigma^{(0)}\supset
\mathcal{V}_\Sigma^{(1)}\supset\cdots\supset \mathcal{V}_\Sigma^{(|\Sigma|)}=\mathcal{V}_\emptyset=\FF_p\cdot1.$$
By convention, we set $\mathcal{V}_\Sigma^{(n)}=\{0\}$ if $n>|\Sigma|$.
We will make use of the polynomials $B_1(\sigma_U):=\prod_{i\in U}(t_i-\theta)$, $U\subset \Sigma$. We note
that for any $U\subset \Sigma$ such that $Nq<|U|<(N+1)q$ there exist polynomials
$a_0^U,\ldots,a_N^U\in \FF_p[\theta][\underline{t}_\Sigma](<q)$, uniquely determined, 
such that
\begin{equation}\label{sumdec}
B_1(\sigma_U)=a_0^U+a_1^U[1]+\cdots+a_N^U[1]^N.
\end{equation}
Also, we need the next Lemma:
\begin{Lemma}\label{derivativesbracket1}
For all $n\geq 1$ and $m\geq 0$, 
$\mathcal{D}_n([1]^m)$ is a polynomial of $\FF_p[\theta^q-\theta]$
of degree $\leq \min\{m-\frac{n}{q},m-1\}$ in $[1]=\theta^q-\theta$.
\end{Lemma}
\begin{proof}
We have $\mathcal{D}_1([1])=-1$, $\mathcal{D}_q([1])=1$ and 
$\mathcal{D}_n([1])=0$ for all $n\geq 1$ with $n\not\in\{1,q\}$.
The statement is thus clear for $m=1$ and the proof can now be obtained by induction 
on $m\geq 1$.
\end{proof}
We note that, for all $n\geq 1$,
\begin{equation}\label{proddec}
\mathcal{D}_n(B_1(\sigma_\Sigma))=(-1)^n\sum_{W\subset\Sigma\atop |\Sigma\setminus W|=n}B_1(\sigma_W).
\end{equation}
This follows from Leibnitz's formula (we set, for simplicity, $\Sigma=\{1,\ldots,s\}$):
\begin{eqnarray*}
\lefteqn{\mathcal{D}_n(B_1(\sigma_\Sigma))=\mathcal{D}_n(\prod_{i\in\Sigma}(t_i-\theta))=}\\
&=&\sum_{i_1+\cdots+i_s=n}\mathcal{D}_{i_1}((t_1-\theta))\cdots \mathcal{D}_{i_s}((t_s-\theta))\\
&=&\sum_{i_1+\cdots+i_s=n\atop 0\leq i_j\leq 1,j=1,\ldots,s}\mathcal{D}_{i_1}((t_1-\theta))\cdots \mathcal{D}_{i_s}((t_s-\theta))\\
&=&(-1)^n\sum_{W\subset\Sigma\atop |\Sigma\setminus W|=n}\prod_{k\in W}(t_k-\theta).\end{eqnarray*}
Our Proposition \ref{dstability} is a direct consequence of the next reinforced statement:
\begin{Proposition}\label{reinforceddstability}
For all $U\subset\Sigma$ and $n,m\geq 0$, we have $\mathcal{D}_n(\mathcal{V}_U^{(m)})\subset
\mathcal{V}_U^{(m+n)}$.
Moreover, if $N=\lfloor \frac{|\Sigma|}{q}\rfloor$, we have 
$a_i^{\Sigma}\in \mathcal{V}^{(qi)}_\Sigma$ for all $i=0,\ldots,N$, where the polynomials $a_i^\Sigma$
are those of the expansion (\ref{sumdec}).
\end{Proposition}
\begin{proof}
We proceed by induction on $s:=|\Sigma|$. In fact, the proof makes use of two nested induction
processes; they are  not complicated, but in order to avoid confusion, we shall refer to the {\em first induction hypothesis}
and to the {\em second induction hypothesis}.
The statement is satisfied for $s=0,1,\ldots,q-1$.
Indeed, in this case, we have $B_1(\sigma_\Sigma)=a_0^\Sigma$ and we know that $B_1(\sigma_\Sigma)=P_\Sigma$. The formula (\ref{proddec}) then implies that $\mathcal{V}_\Sigma$, 
and hence $\mathcal{V}_U$ for all $U\subset\Sigma$, are stable under the operators $\mathcal{D}_1,\ldots,\mathcal{D}_{q-1}$. This implies that $\mathcal{D}_n(\mathcal{V}_U^{(m)})\subset
\mathcal{V}_U^{(m+n)}$ for all $U\subset\Sigma$ and $n,m\geq 0$.

We now suppose that $s\geq q$, so that $N=\lfloor\frac{s}{q}\rfloor\geq 1$. 
We suppose that the statement is satisfied for all $\Sigma'\subsetneq\Sigma$ (this is our first 
induction hypothesis). Let $m$ be an integer 
between $1$ and $N$. We easily verify, by using Leibnitz's formula and using Lemma
\ref{derivativesbracket1}, that
\begin{equation}\label{stepmq}
\mathcal{D}_{mq}(B_1(\sigma_\Sigma))=a_m^\Sigma-\mathcal{D}_{q-1}(a_m^{\Sigma})+
c_0^{\langle m\rangle}+c_1^{\langle m\rangle}[1]+\cdots+c_{N-m}^{\langle m\rangle}[1]^{N-m},
\end{equation}
where, for all $j\in\{0,\ldots,N-m\}$, $c_j^{\langle m\rangle}\in\FF_p[\theta][\underline{t}_\Sigma](<q)$
is an $\FF_p$-linear combination of the polynomials $$a_k^\Sigma,\mathcal{D}_1(a_k^\Sigma),\ldots,\mathcal{D}_{q-1}(a_k^\Sigma),\quad k=m+1,\ldots,N,$$ which are in $\FF_p[\theta][\underline{t}_\Sigma](<q)$.
We also note, again for $m=1,\ldots,N$, by using (\ref{proddec}), that
\begin{equation}\label{stepmqbis}
\mathcal{D}_{mq}(B_1(\sigma_\Sigma))=\sum_{W\subset\Sigma\atop |\Sigma\setminus W|=mq}B_1(\sigma_W),
\end{equation}
and $B_1(\sigma_W)=a_0^W+a_1^W[1]+\cdots+a_{N-m}^W[1]^{N-m}$ by using (\ref{sumdec}).
Since $m>0$, by the (first) induction hypothesis we have that $a_i^W\in\mathcal{V}_W^{(qi)}$ and we can write,
equating the coefficients of $[1]^j$ for all $j$ in (\ref{stepmq}) and (\ref{stepmqbis})
and extracting the constant term,
that
\begin{equation}\label{steptermam}a_m^\Sigma-\mathcal{D}_{q-1}(a_m^{\Sigma})+
c_0^{\langle m\rangle}\in\mathcal{V}_\Sigma^{(mq)}.\end{equation}
The next step is to prove, by induction on $m=N-g$, $g=0,\ldots,N-1$,
that $a_m^\Sigma=a_{N-g}^\Sigma\in\mathcal{V}_\Sigma^{(mq)}$ (this is our second induction process).

We suppose that $m=N$, so that $g=0$. Then, in (\ref{steptermam}), we see that 
$c_0^{\langle N\rangle}=0$. Hence, $a_N^\Sigma-\mathcal{D}_{q-1}(a_N^{\Sigma})\in\mathcal{V}_\Sigma^{(Nq)}$. If $\deg_\theta(a_N^\Sigma)<q-1$
we are done, as in this case, $\mathcal{D}_{q-1}(a_N^{\Sigma})=0$. Otherwise, note that
$\mathcal{D}_{q-1}(\mathcal{D}_{q-1}(a_N^{\Sigma}))=0$. Since
$N>0$, the (first) induction hypothesis implies that $\mathcal{V}_\Sigma^{(Nq)}$ is $\mathcal{D}_{q-1}$-stable
(observe that this space is $\FF_p$-spanned by subspaces $\mathcal{V}_W$ with $W\subsetneq\Sigma$
which are $\mathcal{D}_{n}$-stable for all $n=1,\ldots,q-1$ by the first induction hypothesis, as the various $W$
are such that $|W|<s$). Applying $\mathcal{D}_{q-1}$ to $a_N^\Sigma-\mathcal{D}_{q-1}(a_N^{\Sigma})\in\mathcal{V}_\Sigma^{(Nq)}$ we obtain that $\mathcal{D}_{q-1}(a_N^{\Sigma})\in\mathcal{V}_\Sigma^{(Nq)}$ and summing we get that $a_N^{\Sigma}\in\mathcal{V}_\Sigma^{(Nq)}$ as desired.

The second inductive process is similar (we use the same trick of applying $\mathcal{D}_{q-1}$ as above).
We suppose by (second) induction hypothesis that $a_{m+1}^\Sigma\in\mathcal{V}_\Sigma^{((m+1)q)},\ldots,a_{N}^\Sigma\in\mathcal{V}_\Sigma^{(Nq)}$, so that $a_{m+1}^\Sigma,\ldots,a_{N}^\Sigma\in\mathcal{V}_\Sigma^{((m+1)q)}$
(and we have, by the first induction hypothesis, that $\mathcal{V}_\Sigma^{((m+1)q)}$
is $(\mathcal{D}_{1},\ldots,\mathcal{D}_{q-1})$-stable). In (\ref{steptermam}), we observe that
$c_0^{\langle m\rangle}\in \mathcal{V}_\Sigma^{((m+1)q)}$. Therefore
$a_m^\Sigma-\mathcal{D}_{q-1}(a_m^{\Sigma})\in\mathcal{V}_\Sigma^{(mq)}$ so that, the 
same trick as above yields that $a_m^\Sigma\in\mathcal{V}_\Sigma^{(mq)}$ and this,
for all $m=1,\ldots,N$. For $m=0$ we have observed that $a_0^\Sigma=P_\Sigma$ so our property 
that $a_m^\Sigma\in\mathcal{V}_\Sigma^{(mq)}$ for $m=0,\ldots,N$ is completely checked.

The last step of the proof is to show that $\mathcal{V}_\Sigma$ is $(\mathcal{D}_1,\ldots,\mathcal{D}_{q-1})$-stable and that $\mathcal{D}_n(\mathcal{V}_\Sigma^{(m)})\subset\mathcal{V}_\Sigma^{(m+n)}$. All we need to show, thanks to the first induction hypothesis, is that 
$\mathcal{D}_n(a_0^\Sigma)\in\mathcal{V}_\Sigma^{(n)}$ for $n=1,\ldots,q-1$.
Let $k$ be an integer between $1$ and $q-1$.
We have, by using (\ref{sumdec}) and Lemma \ref{derivativesbracket1}:
\begin{eqnarray*}
\mathcal{D}_k(B_1(\sigma_\Sigma))&=&\mathcal{D}_k(a_0^\Sigma)+\mathcal{D}_k(a_1^\Sigma[1]+\cdots+a_N^\Sigma[1]^N)\\
&=&\mathcal{D}_k(a_0^\Sigma)+e_0^{\langle k\rangle}+e_1^{\langle k\rangle}[1]+\cdots+e_{N-1}^{\langle k\rangle}[1]^{N-1},
\end{eqnarray*}
where the elements $e_j^{\langle k\rangle}$ are polynomials of $\FF_p[\theta][\underline{t}_\Sigma]$ which are $\FF_p$-linear combinations of elements $\mathcal{D}_j(a_l^\Sigma)$ with $j=0,\ldots,q-1$
and $l=1,\ldots,N$. Since these elements are in $\mathcal{V}_\Sigma^{(q)}$ by what we have seen 
above, and since this latter space is $\mathcal{D}_j$-stable for $j=1,\ldots,q-1$ by the first induction hypothesis,
we obtain in particular that $e_0^{\langle k\rangle}\in\mathcal{V}_\Sigma^{(q)}$.
Combining with (\ref{proddec}) and comparing the coefficients of $[1]^l$ for all $l$, we
deduce that $\mathcal{D}_k(a_0^\Sigma)\in\mathcal{V}_\Sigma^{(k)}\subset\mathcal{V}_\Sigma$.
Any element $x$ of $\mathcal{V}_\Sigma$ is a combination $\sum_ix_i$ with $x_i\in \mathcal{V}_{U_i}$.
Hence, $\mathcal{D}_n(\mathcal{V}_\Sigma)\subset\mathcal{V}_\Sigma^{(n)}$ for all $n$.
Let now $x$ be an element of $\mathcal{V}_\Sigma^{(r)}$. Then, $x=\sum_ix_i$ with $x_i\in\mathcal{V}_{U_i}$ and $|\Sigma\setminus U_i|\geq r$, so that $\mathcal{D}_n(x_i)\in\mathcal{V}_{U_i}^{(n)}\subset
\mathcal{V}_\Sigma^{(n+r)}$. We deduce that $\mathcal{D}_n(\mathcal{V}_\Sigma^{(r)})\subset\mathcal{V}_\Sigma^{(n+r)}$.
\end{proof}
We deduce, from the above proof, that for all $U\subset\Sigma$,
$$D_1(P_U)=-\sum_{i\in U}P_{U\setminus\{i\}}+Q,\quad \exists Q\in\mathcal{V}_U^{(2)}.$$

\subsubsection{The case of $d=1$ in the Theorem \ref{simplesumshuffle}: existence of certain coefficients $f_{I,J}$}
We have that ($\Delta$ designates the diagonal subset):
\begin{eqnarray*}
\lefteqn{P_{U,V}:=[1](S_1(1,\sigma_U)S_1(1,\sigma_V)-S_1(2,\sigma_\Sigma))=}\\
&=&[1]\left(\sum_{\lambda,\mu\in\FF_q}\frac{\prod_{i\in U}(t_i-\lambda)}{\theta-\lambda}
\frac{\prod_{j\in V}(t_j-\mu)}{\theta-\mu}-\sum_{\nu\in\FF_q}\frac{\prod_{k\in \Sigma}(t_k-\nu)}{(\theta-\nu)^2}\right)\\
&=&[1]\sum_{(\lambda,\mu)\in\FF_q^2\setminus\Delta}\frac{\prod_{i\in U}(t_i-\lambda)\prod_{j\in V}(t_j-\mu)}{(\theta-\lambda)(\theta-\mu)}\\
&=&
\sum_{(\lambda,\mu)\in\FF_q^2\setminus\Delta}\left(\prod_{i\in U}(t_i-\lambda)\prod_{j\in V}(t_j-\mu)\right)\prod_{\nu\not\in\{\mu,\lambda\}}(\theta-\nu)\in \FF_p[\theta][\underline{t}_\Sigma],\end{eqnarray*}
which is, in particular, a polynomial of degree $\leq q-2$ in $\theta$ (again, we use 
$\operatorname{Gal}(\FF_p^{ac}/\FF_p)$-invariance, which is easily checked, to prove that the coefficients are in $\FF_p$).
We now compute:
\begin{eqnarray*}
\lefteqn{P_{U,V}=[1]\sum_{(\lambda,\mu)\in\FF_q^2\setminus\Delta}\frac{\prod_{i\in U}(t_i-\lambda)\prod_{j\in V}(t_j-\mu)}{(\theta-\lambda)(\theta-\mu)}=}\\
&=&[1]\sum_{\lambda\in\FF_q}\frac{\prod_{i\in U}(t_i-\lambda)}{\theta-\lambda}\sum_{\mu\in\FF_q\setminus\{\lambda\}}\frac{\prod_{j\in V}(t_j-\lambda+\lambda-\mu)}{(\theta-\lambda+\underbrace{\lambda-\mu}_{=:\eta})}\\
&=&[1]\sum_{\lambda\in\FF_q}\frac{\prod_{i\in U}(t_i-\lambda)}{\theta-\lambda}\sum_{\eta\in\FF_q^\times}\frac{\prod_{j\in V}(t_j-\lambda+\eta)}{(\theta-\lambda+\eta)}\\
&=&[1]\sum_{\lambda\in\FF_q}\frac{\prod_{i\in U}(t_i-\lambda)}{\theta-\lambda}\sum_{\eta\in\FF_q^\times}\sum_{M\sqcup W=V}\frac{\prod_{j\in W}(t_j-\lambda)\eta^{|M|}}{(\theta-\lambda)(1+\frac{\eta}{\theta-\lambda})}.
\end{eqnarray*}
Hence, by developing $\frac{1}{1+\frac{\eta}{\theta-\lambda}}$ in $\FF_p((\frac{1}{\theta-\lambda}))$:
\begin{eqnarray*}
P_{U,V}&=&[1]\sum_{\lambda\in\FF_q}\frac{\prod_{i\in U}(t_i-\lambda)}{(\theta-\lambda)^2}\sum_{\eta\in\FF_q^\times}\sum_{M\sqcup W=V}\prod_{j\in W}(t_j-\lambda)\eta^{|M|}\sum_{n\geq 0}(-1)^n\frac{\eta^n}{(\theta-\lambda)^n}\\
&=&[1]\sum_{\lambda\in\FF_q}\prod_{i\in U}(t_i-\lambda)\sum_{M\sqcup W=V}\prod_{j\in W}(t_j-\lambda)\sum_{n\geq 0}\frac{(-1)^n}{(\theta-\lambda)^{n+2}}\sum_{\eta\in\FF_q^\times}\eta^{|M|+n}\\
&=&-[1]\sum_{\begin{smallmatrix} M\sqcup W=V\\ |M|+n> 0\\ |M|+n\equiv0\pmod{q-1}\end{smallmatrix}}(-1)^nS_{1}(n+2,\sigma_{U\sqcup W}).\end{eqnarray*}
We observe that, if $n+2\geq q+1$, then
$\|[1]S_1(n+2,\sigma_U)\|\leq q^{-1}.$ Since $P_{U,V}\in \FF_p[\theta][\underline{t}_\Sigma](<q-1)$,
we thus see that the part of the sum on the right for which $n\geq q-1$ is in the 
maximal ideal $\mathfrak{M}$ of $\FF_p[\underline{t}_\Sigma][[1/\theta]]$. More precisely:
\begin{equation}\label{anotherintermediatestep}
P_{U,V}\equiv -[1]\sum_{\begin{smallmatrix} M\sqcup W=V\\ |M|+n> 0,n<q-1\\ |M|+n\equiv0\pmod{q-1}\end{smallmatrix}}(-1)^nS_{1}(n+2,\sigma_{U\sqcup W})\pmod{\mathfrak{M}}.
\end{equation}
We set, for $U\subset\Sigma$:
$$P_U^{(k)}:=[1]S_1(k,\sigma_U),\quad k\geq 1,$$
so that $P_U^{(1)}=P_U$ in the previous notations.
\begin{Lemma}\label{congruencemodM}
We have the congruence
$$P_U^{(1+n)}\equiv(-1)^n\mathcal{D}_n(P_U)\pmod{\mathfrak{M}},\quad n=0,\ldots,q-1.$$
\end{Lemma}
\begin{proof} If $n=0$, the statement is clear. Assume that $q>n>0$.
By Leibnitz's formula we see that
\begin{eqnarray*}
\mathcal{D}_n(P_U)&=&\mathcal{D}_1([1])\mathcal{D}_{n-1}(S_1(1,\sigma_U))+[1]\mathcal{D}_n(S_1(1,\sigma_U))\\
&=&-\mathcal{D}_{n-1}(S_1(1,\sigma_U))+(-1)^nP_U^{(n+1)}.
\end{eqnarray*}
Now, note that $\|\mathcal{D}_{n-1}(S_1(1,\sigma_U))\|<1$.
\end{proof}
Hence, combining (\ref{anotherintermediatestep}), Lemma \ref{congruencemodM} and Proposition
\ref{dstability}, we see that 
$$P_{U,V}\in\mathcal{V}_\Sigma.$$
We have proved (multiply the above by $[1]^{-1}$) that, given $U,V\subset\Sigma$ such that $\Sigma=U\sqcup V$ there exist, for any decomposition $\Sigma=I\sqcup J$,
an element $f_{I,J}\in\FF_p$, so that 
\begin{equation}\label{withoutcongruences}
S_1(1,\sigma_U)S_1(1,\sigma_V)-S_1(2,\sigma_\Sigma)=\sum_{I\sqcup J=\Sigma}f_{I,J}S_1(1,\sigma_I)\end{equation} (the title of this subsection refers to these coefficients).
We now claim that, if $f_{I,J}\neq0$, then $|I|\equiv |\Sigma|-1\pmod{q-1}$ (that is, $|J|\equiv1\pmod{q-1}$).
For this, we consider, for all $\mu\in\FF_q^\times$, the $\FF_q$-automorphism
$$\psi_\mu:\FF_q(\underline{t}_\Sigma,\theta)\rightarrow\FF_q(\underline{t}_\Sigma,\theta)$$
which sends $t_i$ to $\mu t_i$ and $\theta$ to $\mu\theta$.
Observe that, for all $n\geq 1$,
\begin{eqnarray*}
\psi_\mu(S_1(n,\sigma_\Sigma))&=&\sum_{\lambda\in\FF_q}\frac{\prod_{i\in\Sigma}(\mu t_i-\lambda)}{(\mu\theta-\lambda)^n}\\
&=&\sum_{\lambda\in\FF_q}\frac{\mu^{|\Sigma|}\prod_{i\in\Sigma}\left(t_i-\frac{\lambda}{\mu}\right)}{\mu^n\left(\theta-\frac{\lambda}{\mu}\right)^n}\\
&=&\mu^{|\Sigma|-n}\sum_{\lambda'\in\FF_q}\frac{\prod_{i\in\Sigma}(t_i-\lambda')}{(\theta-\lambda')^n}.
\end{eqnarray*}
Hence, for all $\mu\in\FF_q^\times$ and $I\subset\Sigma$, 
$$\psi_\mu(S_1(1,\sigma_I))=\mu^{|I|-1}S_1(1,\sigma_I).$$
In particular, if $L$ is the left-hand side of the identity (\ref{withoutcongruences}),
we have 
$$\psi_\mu(L)=\mu^{|\Sigma|-2}L,\quad \mu\in\FF_q^\times,$$ and this proves our claim.
For later use, we write the result that we have reached, in the case $d=1$: 
\begin{Proposition}\label{propositiondequal1}
If $\Sigma=U\sqcup V$ there exists, for any decomposition $\Sigma=I\sqcup J$,
an element $f_{I,J}\in\FF_p$, so that 
\begin{multline*}
S_1(1,\sigma_U)S_1(1,\sigma_V)-S_1(2,\sigma_\Sigma)=\\ \sum_{I\sqcup J=\Sigma\atop |I|\equiv|\Sigma|-1\pmod{q-1}}f_{I,J}S_1(1,\sigma_I)=\sum_{I\sqcup J=\Sigma\atop |J|\equiv1\pmod{q-1}}f_{I,J}S_1(1,\sigma_I).\end{multline*}

\end{Proposition}
Note that the congruence conditions on $|J|$ for the $f_{I,J}$ above could have been 
encoded directly in the proof of Proposition \ref{dstability} but we have preferred to 
handle them separately to avoid too technical discussions.

\subsubsection{Computation of the coefficients $f_{I,J}$ in Proposition \ref{propositiondequal1}}

We use the explicit computation of the $\FF_p$-coefficients of the sum-shuffle 
product formula for double zeta values of Chen in \cite{CHE} to compute the coefficients $f_{I,J}$.

\begin{Lemma}\label{basis}
The $\FF_p$-vector space $\mathcal{V}_\Sigma$ has dimension $2^{|\Sigma|}$ and 
the polynomials $P_U$, $U\subset\Sigma$ form a basis of it. 
\end{Lemma}
\begin{proof}
It suffices to show that the elements $S_1(1,\sigma_U)$, $U\subset\Sigma$, are $\FF_p$-linearly independent. First of all, the fractions $S_1(n)=\sum_{\lambda\in\FF_q}\frac{1}{(\theta-\lambda)^n}\in\FF_p(\theta)$, $n=1,2,\ldots$ are linearly independent over $\FF_p$. Indeed, for each $n$,
$S_1(n)$ has poles of order $n$ at each $\theta=\lambda\in\FF_q$. Observe that the set map
$\{U:U\subset\Sigma\}\rightarrow\{0,\ldots,q^s-1\}$ ($\Sigma=\{1,\ldots,s\}$) defined by 
$U\mapsto\sum_{i\in U}q^{i-1}$ is injective (it is bijective if and only if $q=2$).
Thus, the map $\{U:U\subset\Sigma\}\rightarrow\{1,\ldots,q^s\}$ which sends 
$U$ to $q^s-\sum_{i\in U}q^{i-1}$ is also injective; let $G$ be its image.
Then, the map
$$\{S_1(1,\sigma_U):U\subset\Sigma\}\xrightarrow{\psi}\mathcal{W}:=\{S_1(n):n\in G\}$$
defined by (\footnote{We first apply the operator $\tau^s$, $\FF_q[t]$-linear and sending $\theta$ to $\theta^q$, and then, we operate the indicated specialization.})
$$\tau^s(S_1(1,\sigma_U))_{t_i=\theta^{q^{i-1}}\atop i=1,\ldots,s}=S_1\left(q^s-\sum_{i\in U}q^{i-1}\right)$$
is injective, and induces an injective $\FF_p$-linear map $\mathcal{V}_\Sigma\rightarrow\operatorname{Vect}_{\FF_p}(\mathcal{W})$. Since the latter space has dimension $2^s=2^{|\Sigma|}$,
the lemma follows.
\end{proof}
In fact, the map $\psi$ defines a $K$-algebra homomorphism
$$K[\underline{t}_\Sigma]\xrightarrow{\psi} K.$$ In Lemma \ref{basis}, we have seen that 
$\psi$ induces an isomorphism of $\FF_p$-vector spaces 
\begin{multline*}
\mathcal{U}_\Sigma:=\operatorname{Vect}_{\FF_p}(S_1(1,\sigma_U):U\subset\Sigma)\xrightarrow{\psi}\mathcal{W}_\Sigma:=\\ \operatorname{Vect}_{\FF_p}\left(
S_1(n):n=q^n-\sum_{i=1}^sc_iq^{i-1},c_i\in\{0,1\}\right).\end{multline*}
If $U\subset\Sigma$, we write $n_U=q^s-\sum_{i\in U}q^{i-1}$. The map $\psi$ thus sends
$S_1(1,\sigma_U)$ to $S_1(n_U)$.
We have seen that $S_1(1,\sigma_U)S_1(1,\sigma_V)-S_1(2,\sigma_\Sigma)\in\mathcal{U}_\Sigma$,
and we have that 
$$\psi(S_1(1,\sigma_U)S_1(1,\sigma_V)-S_1(2,\sigma_\Sigma))=S_1(n_U)S_1(n_V)-S_1(n_U+n_V),$$
because $\psi(S_1(2,\sigma_\Sigma))=S_1(2q^s-\sum_{i\in \Sigma}q^{i-1})=S_1(q^s-\sum_{i\in U}q^{i-1}+q^s-\sum_{j\in V}q^{j-1})$. Now, we invoke Chen's explicit formula in \cite{CHE} which, we recall,
says that 
\begin{equation}\label{chens}S_1(n_U)S_1(n_V)-S_1(n_U+n_V)=\sum_{0<n<n_U+n_V\atop q-1|n}f_nS_1(n_U+n_V-n),\end{equation}
where $$f_n=(-1)^{n_U-1}\binom{n-1}{n_U-1}+(-1)^{n_V-1}\binom{n-1}{n_V-1}.$$
Assuming that the right-hand side of (\ref{chens}) belongs to $\mathcal{W}_\Sigma$,
it is then equal to a combination
$$\psi\left(\sum_nf_nS_1(1,\sigma_{I_n})\right),\quad I_n\subset\Sigma,$$
and we obtain our result by pulling back this relation, thanks to Lemma \ref{basis}.
So, everything we need to show is that, if $f_n\neq 0$, then  $n=n_J$ for some $J\subset \Sigma$,
and then, compute $f_n$.

We recall that $q=p^e$, $e>0$.
We now write
$$n_U-1=\sum_{k=0}^{es-1}c_k^Up^k,\quad c_k^U\in\{1,\ldots,p-1\},$$
and similarly for $n_V$ etc. so that 
$c_k^U=p-1$ unless $k=e i$ for $i+1\in U$, case in which $c_k^U=p-2$.
Now, let $n$ be in $\{1,\ldots,n_U+n_V-1\}$, $q-1|n$. Since $n_U+n_V-1=2q^s-\sum_{i\in\Sigma}q^{i-1}-2$,
we have that $n<q^s$ for all $q$. We write
$$n-1=\sum_{k=0}^{es-1}c_kp^k,\quad c_k\in\{1,\ldots,p-1\}.$$
Then 
$\binom{n-1}{n_U-1}=\prod_{k=0}^{es-1}\binom{c_k}{c_k^U}$ by Lucas' formula, and 
$\binom{n-1}{n_U-1}\neq 0$ if and only if, for all $k\geq 0$, $c_k\geq c_k^U$. Hence, the latter 
non-vanishing condition is equivalent to $c_k=p-1$ if $e\nmid k$ and if $e\mid k$, then $c_k\in\{p-2,p-1\}$. This means that $\binom{n-1}{n_U-1}\neq 0$ if and only if $n=n_J$
for some $J\subset\Sigma$. In this case, we easily check that 
$\binom{n-1}{n_U-I}=\binom{n_J-1}{n_U-I}=(-1)^{|U\setminus J|}$. Hence,
$$(-1)^{n_U-1}\binom{n_J-1}{n_U-I}=(-1)^{q^s-\sum_{i\in U}q^{i-1}-1}(-1)^{|U\setminus J|}=(-1)^{|U|}(-1)^{|U\setminus J|}=(-1)^{|J|}.$$
Now, observe that $f_n\neq 0$ if and only if either $\binom{n-1}{n_U-1}\neq 0$ or $\binom{n-1}{n_V-1}\neq 0$ and, for what seen above, these two terms cannot be simultaneously non-zero. 
We conclude that 
$f_n\neq 0$ if and only if $n=n_J$ with either $J\subset U$ or $J\subset V$, and if this is the case, 
then $f_n=(-1)^{|J|}$.
Observe also that
if $U\sqcup V=I\sqcup J=\Sigma$, then $n_U+n_V=n_I+n_J$. Hence, $n_U+n_V-n=n_U+n_V-n_J=n_I$
with $I\sqcup J=\Sigma$, and therefore, the right-hand side of (\ref{chens}) is:
$$\sum_{\begin{smallmatrix}0<n<n_U+n_V\\ q-1|n\end{smallmatrix}}f_nS_1(n_U+n_V-n)=-\sum_{
\begin{smallmatrix}I\sqcup J=\Sigma\\ J\subset U \text{ or }J\subset V\\
|J|\equiv 1\pmod{q-1}\end{smallmatrix}}\psi(S_1(1,\sigma_I)),$$ and we are done because $\psi$ 
is injective on $\mathcal{U}_\Sigma$.
We have proved:

\begin{Proposition}\label{propositiondequal1better}
If $\Sigma=U\sqcup V$, then
\begin{multline*}S_1(1,\sigma_U)S_1(1,\sigma_V)-S_1(2,\sigma_\Sigma)=\\ 
=\sum_{\begin{smallmatrix}I\sqcup J=\Sigma\\ |I|\equiv|\Sigma|-1\pmod{q-1}\end{smallmatrix}}f_{I,J}S_1(1,\sigma_I)=-\sum_{\begin{smallmatrix}I\sqcup J=\Sigma\\ |J|\equiv1\pmod{q-1}\\
J\subset U\text{ or }J\subset V\end{smallmatrix}}S_1(1,\sigma_I).\end{multline*}
\end{Proposition}



\subsubsection{The case of $d\geq 1$ in the Theorem \ref{simplesumshuffle}}

In this subsection, we prove Theorem \ref{simplesumshuffle}.
This part follows closely the principles introduced by Thakur in \cite{THA2}.
For the sake of completeness, we give full details.
We denote by $A^+(d)$ the set of monic polynomials of degree $d$ in $A$ and by $A^+(<d)$
the set of monic polynomials of $A$ which have degree $<d$. For $n\in A^+(d)$ and
$m\in A^+(<d)$, we write
$$S_{n,m}=\{(n+\mu m,n+\nu);\mu,\nu\in\FF_q,\mu\neq\nu\}.$$
We have that 
$$S_{n,m}\subset A^+(d)\times A^+(d)\setminus\Delta,$$ where $\Delta$ is the diagonal 
of $A^+(d)\times A^+(d)$. Further, we recall from \cite{THA2}, the next:
\begin{Lemma}
The following properties hold, for $(n,m),(n',m')\in A^+(d)\times A^+(<d)$.
\begin{enumerate}
\item $S_{n,m}\cap S_{n',m'}\neq\emptyset$ if and only if $m=m'$ and $n=n'+\lambda m'$
for some $\lambda\in\FF_q$.
\item If $S_{n,m}\cap S_{n',m'}\neq\emptyset$ then $S_{n,m}=S_{n',m'}$.
\item For all $(a,b)\in A^+(d)\times A^+(d)\setminus\Delta$ there exists 
$(n,m)\in A^+(d)\times A^+(<d)$ with $(a,b)\in S_{n,m}$.
\end{enumerate}
Therefore, the sets $S_{n,m}$ determine a partition of $A^+(d)\times A^+(d)\setminus\Delta$.
\end{Lemma}
\begin{proof} See Thakur's \cite{THA2}.\end{proof}
We also have:
\begin{Lemma}
The following properties hold:
\begin{enumerate}
\item $S'_{n,m}\cap S'_{n',m'}\neq\emptyset$ if and only if 
$m=m'$ and there exists $\mu,\mu'\in\FF_q$ such that
$n+\mu m=n'+\mu'm'$.
\item If $S'_{n,m}\cap S'_{n',m'}\neq\emptyset$ then $S'_{n,m}= S'_{n',m'}$.
\item If $(a,b)\in A^+(d)\times A^+(<d)$, there exists $(n,m)\in A^+(d)\times A^+(<d)$
such that $(a,b)\in S'_{n,m}$.
\end{enumerate}
Therefore, the sets $S'_{n,m}$ determine a partition of $A^+(d)\times A^+(<d)$
And $S'_{n,m}=S'_{n',m'}$ if and only if $S_{n,m}=S_{n',m'}$.
\end{Lemma}
\begin{proof}
Immediate.
\end{proof}
We thus have two partitions:
\begin{eqnarray*}
\mathfrak{S}&=&\{S;S=S_{n,m},(n,m)\in A^+(d)\times A^+(<d)\},\\
\mathfrak{S}'&=&\{S';S'=S'_{n,m},(n,m)\in A^+(d)\times A^+(<d)\}.
\end{eqnarray*}
\subsubsection{End of proof of Theorem \ref{simplesumshuffle}}\label{endofproofoftheorem}
We recall that $\Sigma=U\sqcup V$. We have, for all $d\geq 1$:
\begin{eqnarray*}
S_d(1;\sigma_U)S_d(1;\sigma_V)-S_d(2;\sigma_\Sigma)&=&\sum_{(a,b)\in A^+(d)\times A^+(d)\setminus\Delta}
\frac{\sigma_U(a)\sigma_V(b)}{ab}\\ 
&=&\sum_{S\in\mathfrak{S}}\sum_{(a,b)\in S}\frac{\sigma_U(a)\sigma_V(b)}{ab}.
\end{eqnarray*}
We compute, for any choice of  $S\in\mathfrak{S}$, the sum
$$\sum_{(a,b)\in S}\frac{\sigma_U(a)\sigma_V(b)}{ab}.$$
We note, for all $U\subset\Sigma$, as $\sigma_U$ is a product of injective $\FF_q$-algebra homomorphisms, that it extends in an unique way to a group homomorphism
$\sigma_U:K^\times\rightarrow\FF_q(\underline{t}_\Sigma)^\times.$ In particular, for all $U\subset\Sigma$, we have the identities:
$\sigma_U(n+\mu m)=\sigma_U(m)\sigma_U\left(\frac{n}{m}+\mu\right)$ for all $\mu\in\FF_q$.
We thus have, with $S=S_{n,m}$ and with $\psi_{n,m}$ the substitution $\theta\mapsto\frac{n}{m}$, $t_i\mapsto\chi_{t_i}(\frac{n}{m})$, $i\in\Sigma$:
\begin{eqnarray*}
\sum_{(a,b)\in S}\frac{\sigma_U(a)\sigma_V(b)}{ab}&=&\sum_{\mu,\nu\in\FF_q\atop
\mu\neq \nu}\frac{\sigma_U(n+\mu m)\sigma_V(n+\mu m)}{(n+\mu m)(n+\nu m)}\\
&=&\frac{\sigma_U(m)\sigma_V(m)}{m^{2}}\sum_{\mu,\nu\in\FF_q\atop
\mu\neq \nu}\frac{\sigma_U(\frac{n}{m}+\mu)\sigma_V(\frac{n}{m}+\nu)}{(\frac{n}{m}+\mu)(\frac{n}{m}+\nu)}\\
&=&\frac{\sigma_U(m)\sigma_V(m)}{m^{2}}[S_1(1;\sigma_U)S_1(1;\sigma_V)-S_1(2;\sigma_\Sigma)]_{\psi_{n,m}}\\
&=&-\frac{\sigma_U(m)\sigma_V(m)}{m^{2}}\left[\sum_{\begin{smallmatrix}I\sqcup J=\Sigma\\
|I|\equiv|\Sigma|-1\pmod{q-1}\\ J\subset U\text{ or }J\subset V\end{smallmatrix}}S_1(1;\sigma_{I})\right]_{\psi_{n,m}},
\end{eqnarray*}
by Proposition
\ref{propositiondequal1better}.
Hence,
\begin{eqnarray*}
\sum_{(a,b)\in S}\frac{\sigma_U(a)\sigma_V(b)}{ab}
&=&\frac{\sigma_\Sigma(m)}{m^{2}}\sum_{\begin{smallmatrix}I\sqcup J=\Sigma\\
|I|\equiv|\Sigma|-1\pmod{q-1}\\ J\subset U\text{ or }J\subset V\end{smallmatrix}}\sum_{\mu\in\FF_q}\frac{\sigma_{I}\left(\frac{n}{m}+\mu\right)}{\frac{n}{m}+\mu}\\
&=&\sum_{\begin{smallmatrix}I\sqcup J=\Sigma\\
|I|\equiv|\Sigma|-1\pmod{q-1}\\ J\subset U\text{ or }J\subset V\end{smallmatrix}}\sum_{\mu\in\FF_q}\frac{\sigma_{I}(n+\mu m)\sigma_{J}(m)}{(n+\mu m)m}\\
&=&\sum_{\begin{smallmatrix}I\sqcup J=\Sigma\\
|J|\equiv1\pmod{q-1}\\ J\subset U\text{ or }J\subset V\end{smallmatrix}}\sum_{(a,b)\in S'}\frac{\sigma_{I}(a)\sigma_{J}(b)}{ab},
\end{eqnarray*}
where $S'=S'_{n,m}$. Summing over all $S\in\mathfrak{S}$ induces a sum over all
$S'\in\mathfrak{S}'$. Since
$$\sum_{S'\in\mathfrak{S}'}\sum_{(a,b)\in S'}\frac{\sigma_{I}(a)\sigma_{J}(b)}{ab}=
\sum_{(a,b)\in A^+(d)\times A^+(<d)}\frac{\sigma_{I}(a)\sigma_{J}(b)}{ab}=S_d\left(\begin{matrix}\sigma_I & \sigma_J \\ 1 & 1\end{matrix}\right),$$
and the Theorem follows.

\section{Further properties of the fractions $S_1(1,\sigma_\Sigma)$}

There are completely explicit formulas for $S_1(1,\sigma_\Sigma)$ that 
have not been used yet. One easily proves:
\begin{equation}\label{S0}
S_1(0,\sigma_\Sigma)=\sum_{\lambda\in \FF_q}\prod_{i\in \Sigma}(t_i-\lambda)=
\sum_{j=1}^{\lfloor \frac{|\Sigma|}{q-1}\rfloor}e_{|\Sigma|-j(q-1)}(\underline{t}_\Sigma),
\end{equation} where 
$e_{n}(\underline{t}_\Sigma)$ is the $n$-th elementary polynomial in the variables $\underline{t}_\Sigma$.
Then, we have the following formula, where $W_i=\Sigma\cup V_i$ with $V_i$ a set with exactly 
$i$ elements, such that $\Sigma\cap V_i=\emptyset$, and where $\xi_i$ is the map "substitution
of $t_i$ with $0$" (a $K$-algebra homomorphism), for all $i\in V_i$:
$$S_1(1,\sigma_\Sigma)=\sum_{j=1}^{q-1}(-\theta)^j\xi_{q-1-j}(S_1(0,\sigma_{W_{q-1-j}})).$$
In particular, if $|U|\geq q$ the coefficient of $\theta^{q-1}$ in $S_1(1,\sigma_U)$ (as a polynomial in $\theta$) is proportional to $S_1(0,\sigma_U)$. We have seen, in Lemma \ref{basis}, 
that the fractions $S_1(1,\sigma_U)$, $U\subset\Sigma$ are linearly independent over $\FF_p$.
In the opposite direction, there are non-trivial linear forms between the polynomials $S_1(0,\sigma_U)$.

\begin{Proposition} For all $N\geq 1$ and $b\geq q$, 
The elements $S_1(0,\sigma_V)\in\FF_p[\underline{t}_\Sigma]$ for $V\subset\Sigma$
with $|V|=b$ and $|\Sigma|=b+p^N-1$ are non-zero and linearly dependent over $\FF_p$.
\end{Proposition}

\begin{proof} The non-vanishing is clear because $b\geq q$.
Observe, for $0\leq a\leq b\leq c=|\Sigma|$ (recall
$e_a(\underline{t}_V)=\sum_{U\subset V\atop |U|=a}\prod_{i\in U}t_i$) the relation:
$$\sum_{V\subset\Sigma\atop |V|=b}e_a(\underline{t}_V)=
\sum_{V\subset\Sigma\atop |V|=b}=
\sum_{U\subset V\atop |U|=a}\left(\prod_{i\in U}t_i\right)\sum_{U\subset V\subset\Sigma\atop |V|=b}1=
\binom{c-a}{b-a}\sum_{U\subset V\atop |U|=a}\left(\prod_{i\in U}t_i\right),$$
which holds in any commutative ring $R[\underline{t}_\Sigma]$.
If $v_p(\binom{c-a}{b-a})>0$ (with $v_p$ the $p$-adic valuation of $\QQ$), then this yields a non-trivial $\FF_p$-linear relation among 
the elementary symmetric polynomials $e_a(\underline{t}_V)\in\FF_p[\underline{t}_\Sigma]$ for $V\subset\Sigma$.
Looking at the identity (\ref{S0}), if 
$$
v_p\left(\binom{c-b+j(q-1)}{j(q-1)}\right)>0,\quad \forall j=1,\ldots,\left\lfloor\frac{b}{q-1}\right\rfloor,
$$
 then
\begin{equation}\label{relations0}
\sum_{V\subset \Sigma\atop |V|=b}S_1(0,\sigma_V)=0.
\end{equation}
Since $p-1\mid q-1$, if 
\begin{equation}\label{conditiononvpval} 
v_p\left(\binom{c-b+j(p-1)}{j(p-1)}\right)>0,\quad \forall j=1,\ldots,\left\lfloor\frac{b}{p-1}\right\rfloor,
\end{equation}
then (\ref{relations0}) holds.
We recall that, for $p$ a prime number and $n,k$ integers with $n\geq k\geq 0$,
$v_p(\binom{n}{k})$ is the sum of the carry over in the base-$p$ sum of $k$ and $n-k$.
Hence, $v_p(\binom{n}{k})>0$ if there is at least one carry over.
In (\ref{conditiononvpval}) for fixed $j$, we take $n=c-b+j(p-1)$ and $k=j(p-1)$.
We choose $N\geq 1$ an integer, and we set $c-b=p^N-1=(p-1)\sum_{i=0}^{N-1}p^i$.
We have, at once, $k,n-k<p^N$. Hence, for all $j$, there is carry over in the sum of $k+(n-k)$.
Finally, if we choose $|\Sigma|=c=p^N-1+b$ and $|V|=b$, we are done.
\end{proof}

\section{Linear relations over $K(\underline{t}_\Sigma)$ for few variables}

We denote by $l_d$ the product $(\theta-\theta^{q^d})(\theta-\theta^{q^{d-1}})\cdots (\theta-\theta^{q})$ and by
$b_d(t)$ the product $(t-\theta)(t-\theta^q)\cdots(t-\theta^{q^{d-1}})$.
In particular, with the usual conventions on empty products, $b_0=l_0=1$.

We collect, in this subsection, some explicit linear dependence relation
for our double zeta values in Tate algebras with $|\Sigma|$ small.
First of all, we observe:
\begin{Lemma}
We have:
\begin{equation}\label{rudy1}
S_d(1,\sigma_I)=\frac{\prod_{i\in I}b_d(t_i)}{l_d},\quad I\subset\Sigma=\{1,\ldots,q\},\quad |I|<q.
\end{equation}
\end{Lemma}
\begin{proof}
This formula was first observed by Rudolph Perkins. We can deduce it from the formula 
(5) of \cite{PEL3} (see also the preprint \cite{PEL&PER2}), where it is proved, for $\Sigma=\{1,\ldots,q\}$:
\begin{equation}\label{pelandperk}
F_{d+1}(1,\sigma_\Sigma):=\sum_{i=0}^dS_i(1,\sigma_\Sigma)=l_d^{-1}\prod_{i\in\Sigma}b_d(t_i),\quad d\geq1.\end{equation}
Observe that, since $I\subsetneq\Sigma$, $S_d(1,\sigma_I)\in K[\underline{t}_I]$ is, for $d\geq 1$, the coefficient of 
$\prod_{j\in\Sigma\setminus I}t_j^d$ in the polynomial 
$F_{d+1}(1,\sigma_\Sigma)\in K[\underline{t}_\Sigma]$. From this and from the definition of the polynomials $b_d$, we obtain the formula of the lemma for $d\geq 1$, while the result for $d=0$ is obvious.
\end{proof}
Additionally, we recall from \cite[Lemma 8]{PEL3}, the formula:
 \begin{equation}\label{twistbd}
F_d(1,\chi_t)=\sum_{j=0}^{d-1}S_j(1,\chi_t)=\sum_{j=0}^{d-1}\frac{b_j(t)}{l_j}=\frac{b_d(t)}{(t_i-\theta)l_{d-1}}=\frac{\tau(b_{d-1})(t)}{l_{d-1}},\quad d>0,
\end{equation}
where $\tau$ is the $\FF_q[t]$-linear endomorphism of $K[[t]]$ which associates to a formal series 
$\sum_if_it^i$ the formal series $\sum_if_i^qt^i$.
This formula is easily proved by induction, and occurs in several other references in function field arithmetic.

The simplest example of linear relation is the so-called Euler-Thakur relation, holding 
for $\Sigma=\emptyset$. 
Thakur proved the formula
\begin{equation}\label{Thakurformula}
\zeta_A(1,q-1)=\zeta_A\left(\begin{matrix} \boldsymbol{1} & \boldsymbol{1} \\ 1 & q-1\end{matrix}\right)=\frac{1}{\theta-\theta^q}\zeta_A(q).\end{equation}
This can be viewed, up to certain extent, as an analogue of the famous formula $\zeta(2,1)=\zeta(3)$ by Euler. We recall the proof.
We apply $\tau$ to both left- and right-hand sides of (\ref{twistbd}).
We get the identity
$$\sum_{j=1}^{d-1}\frac{\tau(b_j)(t)}{l_j^q}=\frac{\tau(b_d)(t)}{(t-\theta^q)l_{d-1}^q},\quad d\geq 1.$$
Since $(\tau(b_i)(t))_{t=\theta}=l_i$ for all $i$, we get the identity
$$\sum_{j=0}^{d-1}\frac{1}{l_j^{q-1}}=\frac{l_d}{(\theta-\theta^q)l_{d-1}^q}=\frac{1}{l_{d-1}^{q-1}}.$$
Therefore, with $F_d(q-1)=\sum_{j=0}^{d-1}S_j(q-1)$ and for all $d\geq 1$:
\begin{eqnarray*}
S_d(1,q-1)&=&l_d^{-1}F_d(q-1)\\
&=&=l_{d-1}^{-1}l_{d-1}^{-q+1}\\
&=&\frac{l_{d-1}^{-q}}{\theta-\theta^q}\\
&=&\frac{S_{d-1}(q)}{\theta-\theta^q}.
\end{eqnarray*}
Summing over $d\geq 1$ we get the identity.

We also deduce, from (\ref{twistbd}), for $\Sigma=\{1,\ldots,s\}$ with $s\leq q$:
\begin{equation}\label{formulauseful}
S_d\left(\begin{matrix}\sigma_{\Sigma\setminus\{i\}} & \chi_{t_i} \\
1 & 1\end{matrix}\right)=S_d(1,\sigma_{\Sigma\setminus\{i\}})\sum_{j=0}^{d-1}S_j(1,\chi_{t_i})=\frac{1}{t_i-\theta}\sum_{d\geq 1}\frac{\prod_{j\in\Sigma}b_d(t_j)}{l_dl_{d-1}},\quad i\in\Sigma.
\end{equation}
We immediately obtain:
\begin{Lemma}\label{lemmanewformula}
If $|\Sigma|<q$ and if $i,j$ are distinct elements of $\Sigma$, then\label{anontrivialrelation}
$$\zeta_A\left(\begin{matrix}\sigma_{\Sigma\setminus\{i\}} & \sigma_{\{i\}} \\
1 & 1\end{matrix}\right)=\frac{t_j-\theta}{t_i-\theta}
\zeta_A\left(\begin{matrix}\sigma_{\Sigma\setminus\{j\}} & \sigma_{\{j\}} \\
1 & 1\end{matrix}\right).$$\end{Lemma}
For the next example of linear relation, we have a similar lemma.
\begin{Lemma}\label{anotherformula}
We suppose that $s=|\Sigma|\leq 2q-1$.
We write $\Sigma=U\sqcup V=U'\sqcup V'$ with $|U|=|U'|=q$ and $|V|=|V'|=r$.
Then, we have 
$$\zeta_A\left(\begin{matrix}\sigma_V & \sigma_U\\
1 & q \end{matrix}\right)=\frac{\prod_{i'\in U'}(t_{i'}-\theta)}{\prod_{i\in U}(t_{i}-\theta)}\zeta_A\left(\begin{matrix}\sigma_{V'} & \sigma_{U'}\\
1 & q \end{matrix}\right).$$
\end{Lemma}
\begin{proof}
From (\ref{pelandperk}) we deduce the following formula:
$$F_d(1,\sigma_U):=\sum_{i=0}^{d-1}S_i(1,\sigma_U)=l_{d-1}^{-1}\prod_{i\in U}b_{d-1}(t_i),\quad d\geq 1,$$
and a similar formula holds for $U'$. Therefore
$$F_d(q,\sigma_U):=\tau(F_d(1,\sigma_U))=\sum_{i=0}^{d-1}S_i(q,\sigma_U)=l_{d-1}^{-q}\frac{\prod_{i\in U}b_{d}(t_i)}{\prod_{i\in U}(t_i-\theta)}$$ and analogously for $U'$. Since by the formula (\ref{rudy1})
$$S_d(1,\sigma_V)=\frac{\prod_{i\in V}b_d(t_i)}{l_d}$$ and similarly for $V'$, we get the identity
$$S_d\left(\begin{matrix}\sigma_V & \sigma_U\\
1 & q \end{matrix}\right)\prod_{i\in U}(t_i-\theta)=\frac{\prod_{i\in\Sigma}b_d(t_i)}{l_{d}l_{d-1}^q}=
S_d\left(\begin{matrix}\sigma_{V'} & \sigma_{U'}\\
1 & q \end{matrix}\right)\prod_{i\in U'}(t_i-\theta).$$ Summing over $d\geq 1$ proves the lemma.
\end{proof}

By using Proposition \ref{twists}, note that, for $\Sigma=\{1,\ldots,q\}$:
\begin{multline*}
\tau\left(\zeta_A\left(\begin{matrix} \sigma_{\Sigma\setminus\{i\}} & \sigma_{\{i\}} \\ 1 & 1\end{matrix}\right)\right)_{t_1=\cdots=t_q=\theta}=\zeta_A(1,q-1),\quad i=1,\ldots,q, \\
\tau\left(\zeta_A\left(\begin{matrix} \sigma_{\Sigma} \\ 2 \end{matrix}\right)\right)_{t_1=\cdots=t_q=\theta}=\zeta_A(q).\end{multline*}

It would be nice to see Euler-Thakur's identity arising as specialization of 
a linear relation between the above multiple zeta values, but this does not correspond to the correct intuition. 
We have already seen that the values $\zeta_A\left(\begin{smallmatrix} \sigma_{\Sigma\setminus\{i\}} & \sigma_{\{i\}} \\ 1 & 1 \end{smallmatrix}\right)$ for $i=1,\ldots,q$ generate a $K(\underline{t}_\Sigma)$-subvector space
of dimension $1$. In the opposite direction, we show:
\begin{Lemma}\label{conjecture3}
Assuming that $\Sigma=\{1,\ldots,s\}\subset\{1,\ldots,q\}$ and that $s>0$, the values
$$\zeta_A\left(\begin{matrix} \sigma_{\Sigma\setminus\{s\}} & \chi_{t_s} \\ 1 & 1\end{matrix}\right),\quad\zeta_A\left(\begin{matrix}\sigma_\Sigma \\
2\end{matrix}\right)$$ are linearly independent over $K(\underline{t}_\Sigma)$.
\end{Lemma}
\begin{proof}
We suppose by contradiction that there exist two polynomials $U,V\in A[\underline{t}_\Sigma]$, not both zero, such that 
$$U\zeta_A\left(\begin{matrix} \sigma_{\Sigma\setminus\{s\}} & \chi_{t_s} \\ 1 & 1\end{matrix}\right)=V
\zeta_A\left(\begin{matrix}\sigma_\Sigma \\
2\end{matrix}\right).$$ We can suppose that either $t_s-\theta\nmid U$, or $t_s-\theta\nmid V$.
Evaluating the above identity at $t_s=\theta$ and applying the same techniques of the proof of Proposition \ref{twists} (we only evaluate one variable) we thus get
\begin{equation}\label{before}
(U_{t_s=\theta})(\zeta_A(1,\sigma_{\Sigma\setminus\{s\}})-1)=(V_{t_s=\theta})\zeta_A(1,\sigma_{\Sigma\setminus\{s\}}).\end{equation} This implies that $1,\zeta_A(1,\sigma_{\Sigma\setminus\{s\}})$ are linearly 
dependent over $K(\underline{t}_{\Sigma\setminus\{s\}})$. However, this is impossible. To see this, we again use Proposition \ref{twists} (or rather, the arguments of \S \ref{nonvanishingexample}). We set $X=U_{t_s=\theta}\in K[\underline{t}_{\Sigma\setminus\{s\}}]$
and $Y=V_{t_s=\theta}\in K[\underline{t}_{\Sigma\setminus\{s\}}]$; they are not both identically zero.
Since the subset $\{(\theta^{q^{-k_1}},\ldots,\theta^{q^{-k_{s-1}}}):k_1,\ldots,k_{s-1}\in\NN\}$ is Zariski-dense in the affine space $\mathbb{A}^{s-1}(K^{1/p^\infty})$ (where we recall that $p$ is the characteristic of $\FF_q$ and $K^{1/p^\infty}$ denotes the subfield of $\CC_\infty$ generated by the subfields $\FF_q(\theta^{1/p^k})$, $k\geq 0$), we can choose $k_1,\ldots,k_{s-1}\geq 1$
such that 
$$X_{\begin{smallmatrix}t_i=\theta^{-q^{k_i}}\\ i\in \Sigma\setminus\{s\}\end{smallmatrix}},\quad Y_{\begin{smallmatrix}t_i=\theta^{-q^{k_i}}\\ i\in \Sigma\setminus\{s\}\end{smallmatrix}}\in K^{1/p^\infty}$$ are not both zero.
There exists $N\geq\max\{k_i:i\in\Sigma\setminus\{s\}\}$ an integer such that $m:=q^N-\sum_{i\in\Sigma\setminus\{s\}}q^{N-k_i}>0$. Indeed, Since $k_i\geq 1$ for all $i$, we have $\sum_{i\in\Sigma\setminus\{s\}}q^{N-k_i}\leq |\Sigma\setminus\{s\}|q^{N-1}<q^N$ because of the assumption on the cardinality of $\Sigma$, which is $\leq q$. 

Then, applying the $\FF_q[\underline{t}_{\Sigma\setminus\{s\}}]$-linear operator $\tau^N$ on both right- and left-hand sides of (\ref{before}) we obtain the identity:
$$\tau^N(X)(\zeta_A(q^N,\sigma_{\Sigma\setminus\{s\}})-1)=\tau^N(Y)\zeta_A(q^N,\sigma_{\Sigma\setminus\{s\}}).$$
Substituting $t_i=\theta^{q^{N-k_i}}$ for $i\in\Sigma\setminus\{s\}$ we thus obtain an identity in $K_\infty$:
$$\alpha (\zeta_A(m)-1)=\beta\zeta_A(m),$$ where $$\alpha=\tau^N(X)_{t_i=\theta^{q^{N-k_i}}\atop i\in\Sigma\setminus\{s\}}=\left(X_{t_i=\theta^{q^{-k_i}}\atop i\in\Sigma\setminus\{s\}}\right)^{q^N}\in K$$ and 
similarly, $\beta=\tau^N(Y)_{t_i=\theta^{q^{N-k_i}},i\in\Sigma\setminus\{s\}}\in K$, and $\alpha,\beta$ are not both zero. Now, since $m>0$, the Carlitz zeta value $\zeta_A(m)$ is non-zero. In fact, it is known that it does not belong to $K$ (this is, for instance, a microscopic consequence of \cite[Main Theorem]{CHA&YU}).
But this contradicts the identity obtained, implying that $1,\zeta_A(m)$ are linearly dependent over $K$.
\end{proof}

\begin{Remark}
{\em The Lemma \ref{conjecture3} can be proved in a more elegant 
way after having noticed that $\zeta_A\left(\begin{smallmatrix} \sigma_{\Sigma\setminus\{s\}} & \chi_{t_s} \\ 1 & 1\end{smallmatrix}\right)$ and $\zeta_A\left(\begin{smallmatrix}\sigma_\Sigma \\
2\end{smallmatrix}\right)$ generate a submodule of finite index of
a variant of Taelman's unit module as in \cite{TAE2,APTR} and in the Ph. D. Thesis of
Demeslay \cite{DEM}. Also, the use of deep transcendence results can be avoided by studying carefully sequences of evaluations of these functions which return elements of $K$ the denominators of which can be proved to be unbounded. The details will appear elsewhere in a more general setting.}
\end{Remark}

\subsubsection*{Acknowledgement} The author is thankful to the referee for pertinent suggestions that have contributed to improve the quality of the paper. This research started while the author was visiting the MPIM (Bonn) in February 2016 and the IHES (Bures-sur-Yvette) in April 2016. The author is thankful to both institutions for the very good environment that positively influenced the quality of this work.

\end{document}